\documentclass{ws-ijnt}

\newcommand{\Z}{\mathbb Z}

\newcommand{\fqn}{\mathbb{F}_{q^n}}
\newcommand{\F}{\mathbb{F}}

\newcommand{\q}{\mathcal P}
\newcommand{\m}{\mu_q}
\newcommand{\R}{\rho_q}

\begin{document}

\markboth{Lucas Reis}
{Mean value theorems for a class of density-like arithmetic functions}

%
\catchline{}{}{}{}{}
%

\title{Mean value theorems for a class of density-like arithmetic functions}

\author{LUCAS REIS}
\address{Departamento de Matem\'{a}tica,
Universidade Federal de Minas Gerais,
UFMG,
Belo Horizonte MG (Brazil),
 30123-970\\
\email{lucasreismat@mat.ufmg.br}} 



\maketitle

\begin{history}
\received{(Day Month Year)}
\accepted{(Day Month Year)}
\end{history}

\begin{abstract}
This paper provides a mean value theorem for arithmetic functions $f$ defined by $$f(n)=\prod_{d|n}g(d),$$ where $g$ is an arithmetic function taking values in $(0, 1]$ and satisfying some generic conditions. As an application of our main result, we prove that the density $\m(n)$ (resp. $\R(n)$) of normal (resp. primitive) elements in the finite field extension $\F_{q^n}$ of $\F_q$ are arithmetic functions of (non zero) mean values.
\end{abstract}

\keywords{mean value theorem; arithmetic functions; normal elements; primitive elements; finite fields.}

\ccode{Mathematics Subject Classification 2010: 11H60, 11N37, 11T30}

\section{Introduction}
Given an arithmetic function $f:\mathbb N\to \mathbb R$, one of the most pertinent questions is about the behavior of $f$, {\em on average}. This is measured by the sums $\sum_{n\le x}f(n)$. In particular, if the limit
$$\lim\limits_{x\to +\infty}\frac{1}{x}\sum_{n\le x}f(n),$$
equals $c\in \mathbb R$, we say that $f(n)$ has {\em mean value} $c$. Arithmetic functions may have mean value even if they behave ``irregularly''. For instance, let $\varphi(n)$ be the Euler Totient function and set $F(n)=\frac{\varphi(n)}{n}$. One can show that $\liminf\limits_{n\to +\infty}F(n)=0$ and $\limsup\limits_{n\to+\infty}F(n)=1$ but $F$ has mean value $\frac{6}{\pi^2}$. In fact, there is a more general result on the mean value of {\em multiplicative} functions, i.e., arithmetic functions $f$ such that $f(mn)=f(m)\cdot f(n)$ whenever $\gcd(m, n)=1$. More specifically, if $f$ is multiplicative taking values in $[-1, 1]$, then $f$ has always a mean value, which is equal to $0$ if the series $\sum_{p}\frac{|1-f(p)|}{p}$ diverges~\cite{W67} (this sum is over the prime numbers) and equal to 
$$M_f(\infty)=\prod_{p}\left(1+\frac{f(p)}{p}+\frac{f(p^2)}{p^2}+\ldots \right)\left(1-\frac{1}{p}\right)<+\infty,$$
if the previous series converges~\cite{W44}.

In this paper, we are interested in the average order of arithmetic functions given by convolution products \begin{equation}\label{eq:density-like}f(n)=\prod_{d|n}g(d),\end{equation} where $g$ takes values in $(0, 1]$. We do not assume any further multiplicative property on $g$ or $f$. Our main result, Theorem~\ref{thm:main}, entails that if $g$ satisfies some special conditions, then $f$ possesses a mean value $A_f$ and such mean value can be computed as the limit of a (not uniquely determined) sequence. Moreover, if $g$ is bounded below by a positive (absolute) constant, we prove that $A_f$ is positive. Our main motivation to study this kind of arithmetic functions comes from the density of special elements in finite field extensions. In fact, the density $\R(n)$ of the primitive elements and the density $\m(n)$ of the so called normal elements in the finite field extension $\F_{q^n}$ of $\F_q$ are arithmetic functions given by an identity like Eq.~\eqref{eq:density-like}. As an application of our main result, we prove that $\m(n)$ and $\R(n)$ have positive mean values $\overline{\mu}_q$ and $\overline{\rho}_q$, respectively. The result regarding the existence of $\overline{\rho}_q>0$ is known~\cite{sh}, but the author employed a completely different approach. We also explore the values of $\overline{\mu}_q$ and $\overline{\rho}_q$ as $q$ grows. In particular, good estiamtes on $\overline{\mu}_q$ yield statistical results on the behavior of $\m(n)$, improving previous results; see Theorem~\ref{thm:normal-big}, Corollary~\ref{cor:numerical} and the comments thereafter.

The paper is structured as follows. Section 2 provides background material. In Section 3 we state and prove our main result. Finally, in Section 4, we discuss the applicability of our main result to obtain mean value theorems for the arithmetic functions $\m(n)$ and $\R(n)$.

\section{Preliminaries}
In this short section, we introduce some notation and provide basic background material that is used along the way. For positive integers $a, b$ such that $\gcd(a, b)=1$, let $e_a(b)$ be the order of $a$ modulo $b$, i.e., the least positive integer $k$ such that $a^k\equiv 1\pmod b$. Also, $\varphi(n)$ denotes the Euler Totient function at $n$. As usual, for real valued functions $F$ and $G$, we write $F(x)=O(G(x))$ if $|F(x)|\le C\cdot |G(x)|$ for some absolute constant $C>0$ and write $F(x)=o(G(x))$ if $\lim\limits_{x\to +\infty}\frac{F(x)}{G(x)}=0$.

\subsection{Estimates}
The following lemma provides some inequalities that are frequently used throughout this paper. Its proof is straightforward so we omit details.

\begin{lemma}\label{lem:small}
For any real numbers $x_1, x_2, \ldots, x_n\in [0, 1]$ and  $e_1, e_2, \ldots, e_n\in [1,+\infty)$, we have that
\begin{equation}\label{eq:1-x}\prod_{i=1}^n(1-x_i)^{e_i}\ge 1-\sum_{i=1}^ne_ix_i.\end{equation}
If $0<c<1$ and $x\in [0, c)$, then 
\begin{equation}\label{eq:log2}
|\log(1-x)|\le \alpha_c x,
\end{equation}
for some $\alpha_c>0$ not depending on $x$. Moreover, for $x\in \left(0,\frac{1}{2}\right]$, we have that 
\begin{equation}\label{eq:log}-x^2-x< \log(1-x)\le 0.\end{equation}
\end{lemma}
From the main result in~\cite{NR83}, we have the following lemma.

\begin{lemma}\label{lem:estimate-divisor} If $\sigma_0(m)$ is the number of positive divisors of $m$, then for every $m\ge 3$,
$$\sigma_0(m)<m^{\frac{1.1}{\log\log m}}.$$
\end{lemma}
%

\section{Main result}
Before we state and prove our main result, let us introduce a useful definition.
\begin{definition}
Fix $N$ a positive integer. An arithmetic function $g$ is $N$-density like if  $0< g(n)\le 1$ for any positive integer $n$ with equality $g(n)=1$ whenever $\gcd(n, N)>1$.
\end{definition}
Our main result can be stated as follows.
\begin{theorem}\label{thm:main}
Fix $N$ a positive integer, let $g$ be an $N$-density like arithmetic function such that the series
$$\sum_{d=1}^{\infty}\frac{1-g(d)}{d},$$
converges and set $f(n)=\prod_{d|n}g(d)$. Let $\{L_t\}_{t\ge 1}$ be a sequence of positive integers satisfying the following properties:
\begin{enumerate}[(i)]
\item $L_t$ divides $L_{t+1}$ for $t\ge 1$;
\item $\lim\limits_{t\to+ \infty}L_t=+\infty$;
\item $L_t$ is divisible by every small integer $n\ge 1$ that is relatively prime with $N$. More specifically, there exists a function $h:\mathbb R_{>0}\to \mathbb R$ with $\lim\limits_{x\to +\infty}h(x)=+\infty$ and $h(t)<n$ for every integer $n\ge 1$ not dividing $L_t$ such that $\gcd(n, N)=1$.
\end{enumerate}
Then the sequence $A_t=\frac{(f\ast \varphi)(L_t)}{L_t}=\frac{1}{L_t}\sum_{r|L_t}f(r)\varphi\left(\frac{L_t}{r}\right)$ converges to a limit $A_f\in [0, 1]$ and this is the mean value of $f$, i.e., 
$$A_f=\lim\limits_{x\to +\infty}\frac{1}{x}\sum_{n\le x}f(n).$$
Moreover, if there exists $c>0$ such that $g(d)>c$ for any $d\ge 1$, then the series
$$A_f^*:=\sum_{d=1}^{\infty}\frac{\log g(d)}{d}=\sum_{d=1\atop{\gcd(d, N)=1}}^{\infty}\frac{\log g(d)}{d},$$
converges and 
$$A_f^*=\lim\limits_{x\to +\infty}\frac{1}{x}\sum_{n\le x}\log f(n).$$
In this case, $$A_f\ge \exp(A_f^*)=\prod_{d=1}^{\infty} g(d)^{1/d}>0,$$ and so $f$ has a non zero mean value.
\end{theorem}

\begin{remark}
We observe that the conditions on the number $N$ in Theorem~\ref{thm:main} are not restrictive at all. In fact, any arithmetic function $g$ taking values in $(0, 1]$ is $N$-density like for $N=1$. Moreover, for $N=1$, we can take $L_t$ as the least common multiples of the first $t$ positive integers and $h(t)=t$. 
\end{remark}

\subsection{A note on higher moments}
Fix $N$ a positive integer and $\alpha\ge 1$. If $g$ is an $N$-density like arithmetic function, the same holds for $g^{\alpha}$.  In addition, from Eq.~\eqref{eq:1-x}, for any $\alpha\ge 1$ and any $x\in (0, 1]$, the following inequality holds
$$0\le 1-x^{\alpha}\le \alpha (1-x).$$
In particular, if the series $\sum_{d=1}^{\infty}\frac{1-g(d)}{d}$ converges, so does $\sum_{d=1}^{\infty}\frac{1-g(d)^{\alpha}}{d}$. From these observations, the following corollary follows immediately from Theorem~\ref{thm:main}.

\begin{corollary}\label{cor:main}
Fix $\alpha>1$ and let $g, f$ and $L_t$ be as in Theorem~\ref{thm:main}. Set $A_t^{(\alpha)}=\frac{1}{L_t}\sum_{r|L_t}f(r)^{\alpha}\varphi\left(\frac{L_t}{r}\right)$. Then the sequence $\{A_t^{(\alpha)}\}_{t\ge 1}$ converges to a limit $A_{f}^{(\alpha)}$ and this is the mean value of $f^{\alpha}$, i.e., 
$$A_f^{(\alpha)}=\lim\limits_{x\to +\infty}\frac{1}{x}\sum_{n\le x}f(n)^{\alpha}.$$
In particular, $f$ has variance
$$\sigma(f):=-(A_f)^2+\lim\limits_{x\to +\infty}\frac{1}{x}
\sum_{n\le x}f(n)^2=A_f^{(2)}-(A_f)^2.$$

\end{corollary}

\

\subsection{Proof of Theorem~\ref{thm:main}}
Our proof is divided in three main parts. First, we have the following proposition.

\begin{proposition}\label{prop:crucial}
Let $g, f,  L_t$ and $A_t$ be as in  Theorem~\ref{thm:main}. For $x>0$, let $t=t(x)$ be the unique positive integer such that $L_t^2\le x<L_{t+1}^2$. Then the following holds:
$$\frac{1}{x}\sum_{n\le x}f(n)=A_t+o(1).$$
\end{proposition}

\begin{proof}
We observe that 
$$\sum_{n\le x}f(n)=\underbrace{\sum_{r|L_t}f(r)\sum_{n\le x\atop{\gcd(n, L_t)=r}}1}_{S_1(x)}-\;\underbrace{\sum_{r|L_t}f(r)\sum_{n\le x\atop{\gcd(n, L_t)=r}}\left(1-\frac{f(n)}{f(r)}\right)}_{S_2(x)}.$$
It is direct to verify that, for each divisor $r$ of $L_t$, the number of positive integers $j\le x$ for which $\gcd(j, L_t)=r$ equals $\frac{\varphi\left(\frac{L_t}{r}\right)}{L_t}x+O\left(\varphi\left(\frac{L_t}{r}\right)\right)$ and so we have that
$$S_1(x)-A_tx=O\left(\sum_{r|L_t}\varphi(r)\right)=O(L_t)=o(x),$$ since $L_t\le \sqrt{x}$ and $f(r)\le 1$ for any $r\ge 1$.
It remains to prove that $S_2(x)=o(x)$. Fix $r$ a divisor of $L_t$ and let $n\le x$ be a positive integer such that $\gcd(n, L_t)=r$. Hence, if $d$ divides $n$ but does not divide $r$, we have that $d$ does not divide $L_t$. Let $h:\mathbb R_{>0}\to \mathbb R$ be as in Theorem~\ref{thm:main}.  In particular, any divisor $d$ of $n$ that does not divide $r$ satisfies $d>h(t)$ whenever $\gcd(d, N)=1$. Since $g$ is $N$-density like, we obtain the following inequalities:
\begin{equation}\label{eq:concentration}
\frac{f(n)}{f(r)}=\prod_{d|n\atop d\nmid|r}g(d) \ge \prod_{d|n\atop d>h(t)}(1-(1-g(d)))\ge 1-\sum_{d|n\atop d>h(t)}(1-g(d)),\end{equation}
where in the last inequality we used Eq.~\eqref{eq:1-x}. Since $f(r)\in [0, 1]$ for any $r\ge 1$, Eq.~\eqref{eq:concentration} entails that
$$0\le S_2(x)\le \sum_{n\le x}\sum_{d|n\atop d>h(t)}(1-g(d))=x(1+O(1))\sum_{h(t)<d\le x}\frac{1-g(d)}{d}.$$
So it suffices to show that $$\sum_{h(t)<d\le x}\frac{1-g(d)}{d}=o(1).$$ 
We observe that the previous sum is bounded by $|A(x)-A(h(t))|$, where
$$A(y):=\sum_{1\le d\le y}\frac{1-g(d)}{d}.$$
Since $\{L_t\}_{t\ge 1}$ is non decreasing, it follows that $t=t(x)\to +\infty$ as $x\to+\infty$ and so the same holds for $h(t)$. By an argument of Cauchy sequences, $A(x)-A(h(t))=o(1)$  since the series
$$\sum_{d=1}^{\infty}\frac{1-g(d)}{d}=\lim\limits_{y\to +\infty}A(y),$$
converges. 
\end{proof}
Since $t=t(x)\to +\infty$ as $x\to +\infty$, in order to conclude that $f(n)$ has (finite) mean value, it suffices to prove that the sequence $\{A_t\}_{t\ge 1}$ converges. It follows by the definition that the numbers $A_t$ are non negative. In particular, the sequence $\{A_t\}_{t\ge 1}$ converges if it is nonincreasing. We prove the latter in the following lemma.

\begin{lemma}
Let $g, f,  L_t$ and $A_t$ be as  Theorem~\ref{thm:main}. Then $\{A_t\}_{t\ge 1}$ is a nonincreasing sequence.
\end{lemma}

\begin{proof}
From hypothesis, $L_t$ divides $L_{t+1}$ for any $t\ge 1$. Fix $t\ge 1$ and write $L_{t+1}=bL_t$, where $b$ is a positive integer. We have the following identity
$$A_t-A_{t+1}=\underbrace{\frac{1}{L_t}\sum_{r|L_t}f(r) \frac{b\varphi\left(\frac{L_t}{r}\right)-\varphi\left(\frac{L_{t+1}}{r}\right)}{b}}_{S_1}-\underbrace{\frac{1}{L_t}\sum_{s|L_{t+1}\atop{s\nmid L_t}}f(s) \frac{\varphi\left(\frac{L_{t+1}}{s}\right)}{b}}_{S_2}.$$
For each divisor $r$ of $L_t$, write $b=b_r\cdot b_r^*$, where $b_r$ is the greatest divisor of $b$ such that $\gcd\left(b_r, \frac{L_t}{r}\right)=1$. In particular, $\varphi\left(\frac{L_{t+1}}{r}\right)=b_r^*\cdot\varphi\left(\frac{L_t}{r}\right)\cdot \varphi(b_r)$ and so
$$S_1=\frac{1}{L_t}\sum_{r|L_t}f(r)\varphi\left(\frac{L_t}{r}\right)\cdot \left(1-\frac{\varphi(b_r)}{b_r}\right).$$

We claim that $S_2\le S_1$ and this concludes the proof. We observe that each divisor $s$ of $L_{t+1}$ that does not divide $L_t$ can be written uniquely as $s=ru$ with $r$ a divisor of $L_t$ and $u>1$ a divisor of $b$ such that $\gcd\left(\frac{L_t}{r}, u\right)=1$. In particular, if we write $s$ in this way, we have that $f(s)\le f(r)$ since $f(n)=\prod_{d|n}g(d)$ and $g(d)\in [0, 1]$. Therefore, we obtain the following inequalities
$$0\le S_2\le \frac{1}{L_t}\sum_{r|L_t}f(r)\sum_{1<u|b\atop{\gcd\left(u, \frac{L_t}{r}\right)=1}}\frac{\varphi\left(\frac{L_{t+1}}{ur}\right)}{b}=\frac{1}{L_t}\sum_{r|L_t}f(r)\sum_{1<u|b_r}\frac{\varphi\left(\frac{L_{t+1}}{ur}\right)}{b}.$$
For a divisor $r$ of $L_t$ and a divisor $u$ of $b_r$, we have that $\frac{L_{t+1}}{ur}=\frac{L_t}{r}b_r^*\frac{b_r}{u}$, where the set of prime divisors of $b_r^*$ is contained in the set of prime divisors of $\frac{L_t}{r}$. In particular, we obtain that
$$\frac{\varphi\left(\frac{L_{t+1}}{ur}\right)}{b} = \frac{\varphi\left(\frac{L_{t}}{r}\right)\cdot b_r^*\cdot\varphi\left(\frac{b_r}{u}\right)}{b}=\varphi\left(\frac{L_t}{r}\right)\frac{\varphi\left(\frac{b_r}{u}\right)}{b_r}.$$
Therefore, 
$$0\le S_2\le \frac{1}{L_t}\sum_{r|L_t}f(r)\varphi\left(\frac{L_t}{r}\right)\sum_{u|b_r\atop{u\ne b_r}}\frac{\varphi(u)}{b_r}=S_1.$$
\end{proof}

We proceed to the mean value result for $\log f(n)$. From now and on, we assume that there exists $c>0$ such that $g(d)>c$ for every $d\ge 1$. We observe that $\log f(n)=\sum_{d|n}\log g(d)$ and so $\log f$ is the convolution of $\log g$ and $F\equiv 1$. From this fact, we easily obtain that
$$\sum_{n\le x}\log f(n)=x\sum_{d\le x}\frac{\log g(d)}{d}+O\left(\sum_{d\le x}\log g(d)\right).$$
So it suffices to prove that the series $\sum_{d=1}^{\infty}\frac{\log g(d)}{d}$ converges and that $\sum_{d\le x}\log g(d)=o(x)$.
This is done in the following lemma.

\begin{lemma}
Let $g$ be an arithmetic function taking values in $(c, 1]$, where $0<c<1$. Provided that $\sum_{d=1}^{\infty}\frac{1-g(d)}{d}$ converges, the following hold:

\begin{enumerate}[(i)]
\item $\sum_{d=1}^{\infty}\frac{\log g(d)}{d}$ converges;
\item $\sum_{d\le x}\log g(d)=o(x)$.
\end{enumerate}
\end{lemma}

\begin{proof}
We prove the assertions separately.
\begin{enumerate}[(i)]
\item We observe that $0\le (1-g(d))<1-c$. In particular, Eq.~\eqref{eq:log2} entails that $|\log g(d)|\le \alpha_c \cdot (1-g(d))$ for some absolute constant $\alpha_c$. Since $\sum_{d=1}^{\infty}\frac{1-g(d)}{d}$ converges, it follows that $\sum_{d=1}^{\infty}\frac{\log g(d)}{d}$ converges.

\item  Since $0\le -\log g(d)\le -\log (c)$, we have that 
$$0\le -\sum_{d\le x}\log g(d)\le \underbrace{-\sum_{d\le \sqrt{x}}\log (c)}_{O(\sqrt{x})}-x\left(\sum_{\sqrt{x}<d\le x}\frac{\log g(d)}{d}\right)=o(x),$$
since $\sum_{\sqrt{x}<d\le x}\frac{\log g(d)}{d}=o(1)$ (recall that the series $\sum_{d=1}^{\infty}\frac{\log g(d)}{d}$ converges).
\end{enumerate}
\end{proof}

So it remains to prove that $A_f\ge \exp(A_f^*)$.  Since $g(d)>0$ for every $d\ge 1$, the numbers $f(n)$ are positive. In particular, for every positive integer $N$, the AM-GM inequality yields
$$\frac{1}{N}\sum_{i=1}^Nf(i)\ge \exp\left(\frac{1}{N}\sum_{i=1}^N\log f(i)\right),$$
hence $$0<\exp(A_f^*)\le \limsup\limits_{N\to +\infty}\frac{1}{N}\sum_{i=1}^Nf(i)=A_f.$$

\section{The average density of primitive and normal elements over finite fields}
Throughout this section, $p$ is a prime number, $q=p^m$ is a power of $p$ and $\F_q$ denotes the finite field of $q$ elements. We recall that, up to isomorphism, there exists a unique $n$-degree extension of $\F_q$: such extension has $q^n$ elements and is denoted by $\F_{q^n}$. The field extensions $\F_{q^n}$ have two main algebraic structures. The multiplicative group $\F_{q^n}^*:=\F_{q^n}\setminus\{0\}$ is cyclic and any generator of such group is called {\em primitive}. Moreover, $\F_{q^n}$ (regarded as an $n$-dimensional $\F_q$-vector space) admits a basis $\mathcal C_{\beta}=\{\beta, \beta^q\ldots \beta^{q^n}\}$ comprising the conjugates of an element $\beta\in \F_{q^n}$ by the Galois group $\mathrm{Gal}(\fqn/\F_q)\cong \Z_n$. In this case, such a $\beta$ is called {\em normal } and $\mathcal C_{\beta}$  is a {\em normal basis}. 

Primitive elements are constantly used in cryptographic applications; perhaps, the most notable application is the Diffie-Hellman key exchange~\cite{dh}. Normal bases are also object of interest in applications such as computer algebra, due to their efficiency on basic operations. For instance, if $b=\sum_{i=0}^{n-1}a_i\beta^{q^i}$ and $\mathcal{\beta}$ is a normal element, then $b^q$ is obtained after applying a cyclic shift on the coefficients of $b$ in the basis $\mathcal C_{\beta}$, i.e., $b^q=\sum_{i=0}^{n-1}a_{i-1}\beta^{q^i}$, where the indices are taken modulo $n$. We refer to~\cite{GAO} and the references therein for a nice overview on normal basis, including theory and applications. 

For each positive integer $n$, let $P_q(n)$ and $N_q(n)$ be the number of primitive and normal elements in $\F_{q^n}$, respectively. We observe that the density functions $\R(n):=\frac{P_q(n)}{q^n}$ and $\m(n):=\frac{N_q(n)}{q^n}$ can be viewed as the probability that random element $\alpha\in \F_{q^n}$ is primitive and normal, respectively. It is worth mentioning that some past works explored the behavior of the function $\m(n)$. In~\cite{GP}, the authors provided a lower bound for $\m(n)$ that depends only on the prime factors dividing $n$. Also, $\liminf\limits_{n\to +\infty}\m(n)\sqrt{\log_q n}$ is a positive constant~\cite{F00} and, in particular, $\liminf\limits_{n\to +\infty}\m(n)=0$. In addition, it is well known that $P_q(n)=\varphi(q^n-1)$ and so, by Theorem~15 in~\cite{RS}, we have that $\R(n)\ge \frac{c}{\log n+\log\log q}$ for every $n>1$, where $c>0$ does not depend on $n$ or $q$.

As an application of Theorem~\ref{thm:main}, in this section we prove that such arithmetic functions admit positive mean values and we obtain formulas to compute them. Moreover, we explore the behavior of these mean values as $q\to +\infty$. Our main results can be stated as follows.

\begin{theorem}\label{thm:prim}
For each positive integer $t$, let $L_t$ be the least common multiple of the positive integers $i\le t$. If we set $A_t=\frac{1}{L_t}\sum_{r|L_t}\R(r)\cdot\varphi\left(\frac{L_t}{r}\right)$, then the sequence $\{A_t\}_{t\ge 1}$ converges to a limit $\overline{\rho}_q>0$ and this is the mean value of $\R(n)$, i.e., 
$$\overline{\rho}_q=\lim\limits_{x\to +\infty}\frac{1}{x}\sum_{n\le x}\R(n).$$
Moreover, $\liminf\limits_{q\to +\infty}\overline{\rho}_q=0$.
\end{theorem}

\begin{theorem}\label{thm:normal}
For each positive integer $t$, let $L_t$ be the least common multiple of the numbers $q^i-1$, $i\le t$. If we set $A_t=\frac{1}{L_t}\sum_{r|L_t}\m(r)\cdot\varphi\left(\frac{L_t}{r}\right)$, then the sequence $\{A_t\}_{t\ge 1}$ converges to a limit $\overline{\mu}_q>0$ and this is the mean value of $\m(n)$, i.e., 
$$\overline{\mu}_q=\lim\limits_{x\to +\infty}\frac{1}{x}\sum_{n\le x}\m(n).$$
Moreover, $\lim\limits_{q\to +\infty} \overline{\mu}_q=1$.
\end{theorem}

\subsection{On the density of primitive elements}
Recall that the number of primitive elements in $\F_{q^n}$ equals $\varphi(q^n-1)$. Let $\mathcal P$ be the set of all prime numbers. Here, $\ell$ usually denotes a prime number. We observe that 
$$f_q(n):=\frac{\varphi(q^n-1)}{q^n-1}=\prod_{\ell\in \q\atop {\ell|q^n-1}}\left(1-\frac{1}{\ell}\right).$$
Recall that $e_q(\ell)$ denotes the order of $q=p^m$ modulo $\ell\in \q\setminus\{p\}$, hence $f_q(n)=\prod_{d|n}g_q(d)$, where 
\begin{equation}\label{eq:gi}g_q(d)=\prod_{\ell\in \q\atop e_q(\ell)=d}\left(1-\frac{1}{\ell}\right),\end{equation}
with the convention that $g_q(d)=1$ if the previous product is empty. We also observe that $0<f_q(n)-\R(n)\le \frac{1}{q^n}$ for every $n\ge 1$. Therefore, we have that either both or none of the functions $f_q(n), \R(n)$ possess mean value and, in the affirmative case, such mean values coincide. So we only need to prove Theorem~\ref{thm:prim} replacing $\R(n)$ by $f_q(n)$. 

\subsubsection{Proof of Theorem~\ref{thm:prim}}
We naturally apply Theorem~\ref{thm:main} for $f(n)=f_q(n)$ and $g(n)=g_q(n)$. Some conditions are easily checked. First, we observe that if $L_t$ denotes the  least common multiple of the positive integers $i\le t$, then $L_t$ divides $L_{t+1}$ and $L_t\ge t$ for any $t\ge 1$. In particular, $\lim\limits_{t\to +\infty}L_t= +\infty$. From definition, the function $g_q(d)$ is $1$-density like. We observe that, for any positive integer $d$ that does not divide $L_t$, we have that $d>t$ and so the function $h$ in Theorem~\ref{thm:main} can be taken as $h(t)=t$. All in all, in order to prove that $\rho_q(n)$ has positive mean value, it suffices to check the following:
\begin{enumerate}
\item the series $\sum_{d=1}^{\infty}\frac{1-g_q(d)}{d}$ converges;
\item there exists  $0<c<1$ such that $g_q(d)>c$ for every $d\ge 1$.
\end{enumerate}
We prove the latter in the following lemma.
\begin{lemma}
Let $g_q(d)$ be as in Eq.~\eqref{eq:gi}. Then, for every integer $d\ge 2$, we have that 
$$0\le 1-g_q(d)\le \sum_{\ell\in \q\atop e_q(\ell)=d}\frac{1}{\ell}=O\left(\frac{\log d}{d}\right).$$
In particular, $\lim\limits_{d\to +\infty} g_q(d)=1$ and the series $\sum_{d=1}^{\infty}\frac{1-g_q(d)}{d}$ converges. 
\end{lemma}
\begin{proof}
If there is no prime $\ell\in \q$ such that $e_q(\ell)=d$, we obtain that $1-g_q(d)=0$. Otherwise, let $\ell_1^{(d)}< \cdots< \ell_{u(d)}^{(d)}$ be the primes $\ell$ such that $e_q(\ell)=d$. In particular, $d$ divides $\varphi(\ell_i^{(d)})=\ell_i^{(d)}-1$ and so $\ell_i^{(d)}\ge di+1$. Therefore, we have that
$$q^d>\prod_{\ell\in \q\atop{e_q(\ell)=d}}\ell>d^{u(d)}=q^{u(d)\cdot \log_q d},$$
hence $u(d)=O(d)$. From Eq.~\eqref{eq:1-x}, for $d\ge 2$, we have that
$$1-g_q(d)\le \sum_{\ell\in \q\atop{e_q(\ell)=d}}\frac{1}{\ell}< \sum_{i=1}^{u(d)}\frac{1}{di}=O\left(\frac{\log d}{d}\right).$$
In particular, $\lim\limits_{d\to +\infty} g_q(d)=1$ and $\frac{1-g_q(d)}{d}=O\left(\frac{\log d}{d^2}\right)$ for $d\ge 2$. Since the series $\sum_{d=1}^{\infty}\frac{\log d}{d^2}$ converges, the same holds for $\sum_{d=1}^{\infty}\frac{1-g(d)}{d}$.

\end{proof}

It remains to prove that $\liminf\limits_{q\to +\infty} \overline{\rho}_q=0$. We observe that, since $q-1$ divides $q^n-1$ for any positive integer $n$, we have that $\overline{\rho}_q\le \frac{\varphi(q-1)}{q-1}$. We obtain the following result.

\begin{proposition}
Let $q$ be a power of a prime $p$. Then $\liminf\limits_{i\to +\infty} \overline{\rho}_{q^i}=0$.
\end{proposition}

\begin{proof}
Recall that $\rho_{q^i}\le \frac{\varphi(q^{i}-1)}{q^{i}-1}$ for every integer $i\ge 1$. Let $\alpha_k$ be the product of the first $k$ prime numbers, distinct from $p$, and let $e_k$ be the least positive integer such that $q^{e_k}-1$ is divisible by $\alpha_k$. Therefore, 
$$\overline{\rho}_{q^{e_k}}\le \frac{\varphi(q^{e_k}-1)}{q^{e_k}-1}\le \prod_{\ell \in \q\atop \ell |\alpha_k}\left(1-\frac{1}{\ell}\right).$$
Since $\prod_{\ell\in \q}\left(1-\frac{1}{\ell}\right)=0$, we have that $\lim\limits_{k\to +\infty }\prod_{\ell \in \q\atop \ell |\alpha_k}\left(1-\frac{1}{\ell}\right)=0$, from where the result follows.
\end{proof}

\subsection{On the density of normal elements}
We recall that the number of primitive elements in a finite field is given implicitly by the Euler Totient function. There is an analogue of such function for polynomials over finite fields and this analog function plays an important role in counting normal elements.  The Euler Totient function  $\Phi_q$ for polynomials over $\F_q$ is defined as follows: for an irreducible polynomial $g\in \F_q[x]$ of degree $r$ and a positive integer $m$, we set $\Phi_q(g^m)=q^{(m-1)r}(q^r-1)$ and then $\Phi_q$ extends multiplicatively. We have the following result.
\begin{theorem}[\cite{LiNi} Theorem~3.73]
\label{thm:normal-1}
For any positive integer $n$, the number of elements in $\F_{q^n}$ that are normal over $\F_q$ equals $\Phi_q(x^n-1)$.
\end{theorem}
Recall that for any positive integer $d$, not divisible by $p$, the $d$-th cyclotomic polynomial is defined as $E_d(x)=\prod_{\alpha\in \Omega(d)}(x-\alpha)$, where $\Omega(d)$ is the set of primitive $d$-th roots of unity. The degree of $E_d$ equals $\varphi(d)$ and, over finite fields, $E_d$ has nice factorization.
\begin{lemma}[\cite{LiNi} Theorem 2.47]\label{lem:cyclo}
For any positive integer $d$ that is not divisible by $p$, $E_d(x)$ factors into $\frac{\varphi(d)}{e_q(d)}$ distinct irreducible polynomials over $\F_q$, each of degree $e_q(d)$.
\end{lemma}
The cyclotomic polynomials satisfy the recursive identity $x^n-1=\prod_{d|n}E_d(x)$ if $\gcd(n, p)=1$. In addition, for $n=p^u\cdot M$ with $\gcd(M, p)=1$, we have that $x^n-1=\prod_{d|M}E_d(x)^{p^u}$. Therefore, from Theorem~\ref{thm:normal-1} and Lemma~\ref{lem:cyclo}, we have the following identity for an arbitrary integer $n\ge 1$ not divisible by $p$:
\begin{equation}\label{eq:normal}
\m(n)=\prod_{d|n}\left(1-\frac{1}{q^{e_q(d)}}\right)^{\frac{\varphi(d)}{e_q(d)}}.
\end{equation}
In addition, $\m(n)=\m(np^j)$ for every $j\ge 0$. Therefore, if we set

\begin{equation}\label{eq:GI}G_q(d)=\begin{cases}\left(1-\frac{1}{q^{e_q(d)}}\right)^{\frac{\varphi(d)}{e_q(d)}}&\text{if}\; \gcd(p, d)=1,\\
1&\text{otherwise,}\end{cases}\end{equation}
we have that $\m(n)=\prod_{d|n}G_q(d)$. 

\subsubsection{Proof of Theorem~\ref{thm:normal}} We prove Theorem~\ref{thm:normal} by applying Theorem~\ref{thm:main} with $f(n)=\mu_q(n)$ and $g(n)=G_q(n)$. Some conditions are easy to check. From construction, $G_q(d)$ is $p$-density like, where $p$ is the characteristic of $\F_q$.  Let $L_t$ be the least common multiple of the numbers $q^i-1$ with $i\le t$. In particular, $L_t$ divides $L_{t+1}$ and $L_t\ge q^t-1$ for any $t\ge 1$.  Therefore, $\lim\limits_{t\to +\infty}L_t= +\infty$. We observe that $e_q(d)>t$ for every positive integer $d$ that is relatively prime with $p$ and does not divide $L_t$. Since $e_q(d)\le d$,  the function $h$ in Theorem~\ref{thm:main} can be taken as $h(t)=t$. In overall, in order to prove that $\rho_q(n)$ has positive mean value, it suffices to check the following:
\begin{enumerate}
\item the series $\sum_{d=1}^{\infty}\frac{1-G_q(d)}{d}$ converges;
\item there exists  $0<c<1$ such that $G_q(d)>c$ for every $d\ge 1$.
\end{enumerate}
We prove the latter in the following lemma.
\begin{lemma}\label{lem:aux-normal}
Let $G_q(d)$ be as in Eq.~\eqref{eq:gi}. Then, for any $d\ge 1$ with $\gcd(d, p)=1$ we have that
$$0\le 1-G_q(d)\le \frac{\varphi(d)}{q^{e_q(d)}e_q(d)}\le \frac{\log q}{\log (d+1)}.$$
In particular, $\lim\limits_{d\to +\infty} G_q(d)=1$ and the series $\sum_{d=1}^{\infty}\frac{1-G_q(d)}{d}$ converges.
\end{lemma}

\begin{proof}
 If $\gcd(d, p)=1$,  Eq.~\eqref{eq:1-x} entails that $G_q(d)\ge 1-\frac{\varphi(d)}{e_q(d)q^{e_q(d)}}$ and so
$$0\le 1-G_q(d)\le \frac{\varphi(d)}{e_q(d)q^{e_q(d)}}\le \frac{\log q}{\log (d+1)},$$
for every $d\ge 1$ since $\varphi(d)\le d<q^{e_q(d)}$. We observe that the numbers $1-G_q(d)$ are non negative and $1-G_q(d)=0$ if $\gcd(d, p)>1$. In particular, we have that
\begin{align*}0\le \sum_{d=1}^{\infty}\frac{1-G_q(d)}{d}= &\sum_{d=1\atop{\gcd(d, P)=1}}^{\infty}\frac{1-G_q(d)}{d}\le \sum_{d=1\atop{\gcd(d, P)=1}}^{\infty} \frac{\varphi(d)}{de_q(d)q^{e_q(d)}}=\\ {}&\sum_{j=1}^{\infty}\frac{1}{jq^j}\sum_{e_q(d)=j}\frac{\varphi(d)}{d}\le \sum_{j=1}^{\infty}\frac{\sigma_0(q^j-1)}{jq^j},\end{align*}
where in the last inequality we used the fact that $\varphi(d)\le d$ and that the number of positive integers $d$ with $e_q(d)=j$ is at most the number $\sigma_0(q^j-1)$ of positive divisors of $q^j-1$. From Lemma~\ref{lem:estimate-divisor}, $\sigma_0(q^j-1)=O(q^{j/2})$ and so
$$\sum_{j=1}^{\infty}\frac{\sigma_0(q^j-1)}{jq^j}<+\infty.$$
\end{proof}
It remains to prove that $\overline{\mu}_q\to 1$ as $q\to +\infty$. In fact, we prove something stronger.

\begin{theorem}\label{thm:normal-big}
If $q\ge 4$ is a prime power and $\overline{\mu}_q$ is the mean value of $\mu_q(n)$, then
$$1-\frac{1}{q}-\frac{1}{\sqrt{q}}< \overline{\mu}_q\le 1-\frac{1}{q},$$
%
\end{theorem}

\begin{proof}
We have proved that $\overline{\mu}_q$ exists and it is positive. From the trivial bound $\mu_q(n)\le 1-\frac{1}{q}$, we conclude that $\overline{\mu}_q\le 1-\frac{1}{q}$. Moreover, Theorem~\ref{thm:main} and Eq.~\eqref{eq:GI} entail that, for $s_q(j)=\sum_{e_q(d)=j}\frac{\varphi(d)}{d}$, the following holds
$$\overline{\mu}_q\ge \prod_{j=1}^{\infty}\left(1-\frac{1}{q^j}\right)^{\frac{s_q(j)}{j}}.$$
We observe that $(1-x)^t\ge 1-x$ for every $x, t\in (0, 1)$. The latter, combined with Eq.~\eqref{eq:1-x}, yields the following inequality $$\prod_{j=1}^{\infty}\left(1-\frac{1}{q^j}\right)^{\frac{s_q(j)}{j}}\ge 1-\sum_{j=1}^{\infty}\frac{\delta_q(j)}{q^j},$$
where $\delta_q(j)=\max\left\{1, \frac{s_q(j)}{j}\right\}$. Since $\varphi(d)\le d$ for every $d\ge 1$, it follows that $s_q(j)$ is at most $\sigma_0(q^j-1)$, the number of positive divisors of $q^j-1$. In particular, Lemma~\ref{lem:estimate-divisor} entails that $s_q(j)\le q^{j/2}$ whenever $q^j\ge 10^4$. By a direct computation, we verify that the same holds in the range $q^j\le 10^4$.
It follows by induction on $j$ that $q^{j/2}\ge j$ if $q\ge 4$ and $j\ge 1$. In particular, $\delta_q(j)\le \frac{q^{j/2}}{j}$ and so 
$$-\sum_{j=1}^{\infty}\frac{\delta_q(j)}{q^j}\ge -\sum_{j=1}^{\infty}\frac{1}{jq^{j/2}}= \log\left(1-\frac{1}{\sqrt{q}}\right)> -\frac{1}{q}-\frac{1}{\sqrt{q}}\;\text{for}\;\; q\ge 4,$$
where the last inequality follows by Eq.~\eqref{eq:log}.
\end{proof}

The previous theorem also implies an estimate on the variance of $\mu_q(n)$. In fact, as shown in the following corollary, this variance goes to zero as $q\to +\infty$.

\begin{corollary}
The limit $\mu_q^{(2)}:=\lim\limits_{x\to \infty}\frac{1}{x}\sum_{1\le n\le x}\mu_q(n)^2$ exists. In particular, for $q\ge 4$, the variance $\sigma(\mu_q):=\mu_q^{(2)}-(\overline{\mu}_q)^2$ of $\mu_q(n)$ satisfies 
$$0\le \sigma(\mu_q)< \frac{2}{\sqrt{q}}-\frac{1}{q}-\frac{2}{q\sqrt{q}}.$$
\end{corollary}
\begin{proof}
The existence of $\mu_q^{(2)}$ is an immediate consequence of Corollary~\ref{cor:main}. We have the trivial bound $\mu_q^{(2)}\le \left(1-\frac{1}{q}\right)^2$ and Theorem~\ref{thm:normal-big} provides the bound $\overline{\mu}_q> 1-\frac{1}{q}-\frac{1}{\sqrt{q}}$, from where the result follows.

\end{proof}

\subsubsection{Numerical results on the density of normal elements}
We comment on previous results on lower bounds for the function $\mu_q(n)$, the density of normal elements in $\F_{q^n}$ (over $\F_q$). According to Theorem~3 of~\cite{F00}, there exists $c>0$ such that
$$\mu_q(n)\ge .28477\frac{1}{\sqrt{\log_q n}}, \; \text{for}\; q\ge 2, n\ge q^c.$$
Although this bound holds for every prime power $q$, the function $\frac{1}{\sqrt{\log_q n}}$ goes to zero as $n$ goes to infinity. Taking into account the set of prime divisors of $n$, it is possible to obtain positive lower bounds. In fact, according to Theorem~3.3 of~\cite{GP}, if we fix a set $S$ of distinct prime numbers $p_1, \ldots, p_s$, there exists a constant $C=C(S)>0$ such that
$$\mu_q(n)>C,$$
whenever $S$ is the set of distinct prime factors of $n$. However, such a constant is not given explicitly. Our aim here is to apply Theorem~\ref{thm:normal-big} in order to obtain explicit numerical results, for not all $n$ but considerable proportion of $\mathbb N$ (in the sense of natural density). This is done in the following corollary.

\begin{corollary}\label{cor:numerical}
Let $q\ge 4$ be a prime power and fix $T>0$. Then here exists a constant $C>0$ such that, if $x\ge C$, for all but at most $\frac{x}{1+T\sqrt{q}}$ positive integers $n\le x$, we have that $\mu_q(n)\ge C_{q, T}:=1-\frac{1}{q}-\frac{1}{\sqrt{q}}-T$. 
\end{corollary}

\begin{proof}
We observe that, from Theorem~\ref{thm:normal-big}, there exists $C>0$ such that $S_q(x)=\frac{1}{x}\sum_{n\le x} \mu(n)\ge 1-\frac{1}{q}-\frac{1}{\sqrt{q}}$ whenever $x\ge C$. Fix $x\ge C$ and suppose that $\delta x$ positive integers $n\le x$ are such that $\mu_q(n)<C_{q, T}$. Since the bound $\mu_q(n)\le 1-\frac{1}{q}$ trivially holds for every $n\ge 1$, we conclude that
$$S_q(x)\le \delta C_{q, T}+(1-\delta)(1-1/q),$$
hence $\delta\left(\frac{1}{\sqrt{q}}+T\right)\le \frac{1}{\sqrt{q}}$ and the result follows.
\end{proof}

The previous corollary yields some explicit numerical results: for instance, at least half of the positive integers $n$ are such that $\mu_q(n)\ge 1-\frac{1}{q}-\frac{2}{\sqrt{q}}$, which is positive for $q\ge 7$. Moreover, such bound goes fast to $1$: we have that $$1-\frac{1}{q}-\frac{2}{\sqrt{q}}\ge .95,$$ 
for $q\ge 1,640$.

\section*{Acknowledgments}
We thank the anonymous reviewers for many helpful comments and suggestions. The author was supported by FAPESP under grant 2018/03038-2, Brazil.




\end{document}